\documentclass{amsart}
\usepackage[mathscr]{eucal}

\newtheorem{theorem}{Theorem}[section]
\newtheorem{lemma}[theorem]{Lemma}

\theoremstyle{definition}

\newtheorem{cor}[theorem]{Corollary}

\theoremstyle{remark}
\newtheorem{remark}[theorem]{Remark}

\numberwithin{equation}{section}

\begin{document}
	%.....................................
	
	\title[ Joint numerical radius for tuples] {{ Joint numerical radius of Tuples: Extreme points, subdifferential set and Gateaux derivative }}
	
	%	\title[Gateaux derivative of joint numerical radius for tuples] {{ Gateaux derivative of joint numerical radius for tuples}}
	
	%\author[]{}
	
\author[Arpita Mal]{Arpita Mal}

\address[]{Dhirubhai Ambani University\\ Gandhinagar-382007\\ India.}
\email{arpitamalju@gmail.com}

%\thanks will become a 1st page footnote.
\thanks{ The author would like to thank DST, Govt. of India for the financial support in the form of Inspire Faculty Fellowship (DST/INSPIRE/04/2022/001207).}


%    General info


%    General info

\date{}
\maketitle
\begin{abstract}
Suppose $\mathcal{Z}$ is the space of all tuples of operators on a finite-dimensional Banach space endowed with the joint numerical radius norm. We obtain the structure of the extreme points of the dual unit ball of $\mathcal{Z}.$ Using this, we derive an expression for the subdifferential set of the joint numerical radius of a tuple in $\mathcal{Z}.$ Applying this expression, we characterize smooth tuples and Birkhoff-James orthogonality in $\mathcal{Z}.$ Finally, we obtain the Gateaux derivative of the  joint numerical radius of a tuple.  
\end{abstract}

\section{Introduction}
The purpose of this article is to derive an expression for the Gateaux derivative of the  joint numerical radius of a tuple of operators on a finite-dimensional Banach space. The main tool for this is an expression for the set of supporting linear functionals (which is precisely the subdifferential set) of a tuple of operators on a finite-dimensional Banach space, equipped with the joint numerical radius. To proceed further, let us introduce the necessary notations and terminologies.\\
Throughout the article, the symbol $\mathcal{X}$ denotes the Banach space $(\mathcal{X},\|\cdot \|)$ over the field $\mathbb{F}$ and $\mathcal{X}^*$ denotes its dual space. Suppose $B_\mathcal{X}$ and $S_{\mathcal{X}}$ are the closed unit ball and the unit sphere of $\mathcal{X}$ respectively. For $x,y\in \mathcal{X},$ the right-hand and left-hand Gateaux derivatives of $x$ in the direction of $y$ are denoted by $G_+(x,y)$ and $G_-(x,y),$ respectively and defined as 
\[G_+(x,y)=\lim_{t\to 0+}\frac{\|x+ty\|-\|x\|}{t},\quad G_-(x,y)=\lim_{t\to 0-}\frac{\|x+ty\|-\|x\|}{t}.\]
It is well-known that both the limits exist. The norm of $\mathcal{X}$ is said to be Gateaux differentiable at $x$ if $G_+(x,y)=G_-(x,y)$ holds for all $y\in \mathcal{X}.$  In this case, the Gateaux derivatives of $x$ in the direction of $y$ is $G(x,y)=G_+(x,y)=G_-(x,y).$ Recall that a non-zero vector $x$ is said to be smooth if the set $$J(x):=\{f\in \mathcal{X}^*:\|f\|=1,f(x)=\|x\|\}$$ is singleton. It is well-known that $x$ is smooth if and only if the norm is Gateaux differentiable at $x,$ and in this case, 
\[G(x,y)=\Re\{f(y)\}, \text{ where } J(x)=\{f\}.\] On the other hand, for arbitrary $x,y\in \mathcal{X},$
\begin{equation}\label{eq-03}
	\begin{split}
	G_+(x,y)=&\max\{\Re\{f(y)\}:f\in J(x)\}, \text{ and } \\
	G_-(x,y)=&\min\{\Re\{f(y)\}:f\in J(x)\}.
	\end{split}
\end{equation} 
In general, an element of $J(x)$ is called a supporting linear functional of $x.$ It is easy to observe that $f\in J(x)$ if and only if $\|y\|-\|x\|\geq \Re\{f(y-x)\}$ for all $y\in \mathcal{X}.$ Therefore, $J(x)$ is also known as the subdifferential set of the norm at $x.$ Over the years, Gateaux derivative and smoothness of operators on a Banach space have been studied (see \cite{A,H,K,MPS,S}).  Another area of independent interest is the numerical radius of operators. Suppose $T$ is a bounded linear operator on a Hilbert space $\mathcal{H}.$ Then the numerical radius $w(T)$ of $T$ is defined as 
\[w(T):=\sup\{|\langle Tx,x\rangle|:x\in S_\mathcal{H}\}.\]
Analogously, if $T$ is a bounded linear operator on $\mathcal{X},$ then the numerical radius $w(T)$ of $T$ is 
\[w(T):=\sup\{|x^*(Tx)|:(x,x^*)\in \Pi \},\]
 where $\Pi=\{(x,x^*):x^*\in S_{\mathcal{X}^*},x\in S_\mathcal{X},x^*(x)=1\}.$ Suppose $\mathcal{L}(\mathcal{X})$ is the space of all bounded linear operators on $\mathcal{X}.$ It is well-known that if $\mathcal{X}$ is a complex Banach space, then $w(\cdot)$ is a norm on $\mathcal{L}(\mathcal{X}).$ There are some real Banach spaces also such that $w(\cdot)$ is a norm on $\mathcal{L}(\mathcal{X}).$ Throughout the article, without mentioning further, we only consider those Banach spaces $\mathcal{X}$ such that $w(\cdot)$ is a norm on $\mathcal{L}(\mathcal{X}).$ Suppose $T_1,\ldots,T_d\in \mathcal{L}(\mathcal{H}).$ Consider the tuple $\mathcal{T}=(T_1,\ldots,T_d):\mathcal{H}\to \oplus_{i=1}^d\mathcal{H}$ defined as $\mathcal{T}x=(T_1x,\ldots,T_dx).$ For $1<p<\infty,$ the joint numerical radius $w_p(\mathcal{T})$ of the tuple $\mathcal{T}$ is defined as 
 \[w_p(\mathcal{T}):=\sup\bigg\{\Big(\sum_{i=1}^d|\langle T_ix,x\rangle|^p\Big)^\frac{1}{p}:x\in S_{\mathcal{H}}\bigg\}.\]
 Analogously, if $T_1,\ldots,T_d\in \mathcal{L}(\mathcal{X}),$ then for the tuple $\mathcal{T}=(T_1,\ldots,T_d):\mathcal{X}\to \oplus_{i=1}^d\mathcal{X},$ %defined as $\mathcal{T}x=(T_1x,\ldots,T_dx),$ and for $1<p<\infty,$ the joint numerical radius $w_p(\mathcal{T})$ of the tuple $\mathcal{T}$ is defined as 
 $w_p(\mathcal{T})$ is defined in \cite{M} as
 \[w_p(\mathcal{T}):=\sup\bigg\{\Big(\sum_{i=1}^d|x^*( T_ix)|^p\Big)^\frac{1}{p}:(x,x^*)\in \Pi\bigg\}.\]
 It is easy to observe that the joint numerical radius $w_p(\cdot)$ defines a norm on the space of tuples of operators. Suppose $\mathcal{L}(\mathcal{X})^d$ denotes the space of all $d$-tuples of bounded linear operators on $\mathcal{X}.$ Unless otherwise mentioned, we always assume that $\mathcal{Z}=(\mathcal{L}(\mathcal{X})^d,w_p(\cdot)),$ the space $\mathcal{L}(\mathcal{X})^d$ equipped with the joint numerical radius. For a tuple $\mathcal{T}\in \mathcal{Z},$ we denote the set of supporting linear functionals of $\mathcal{T}$ by $J_{w_p}(\mathcal{T}),$ that is,
 \[J_{w_p}(\mathcal{T}):=\{f\in S_{\mathcal{Z}^*}:f(\mathcal{T})=w_p(\mathcal{T})\}.\]
 In \cite{GS}, the authors studied the set $J_{w_2}(\mathcal{T}),$ when $\mathcal{T}\in (\mathcal{L}(\mathcal{H})^d,w_2(\cdot))$ and $\dim(\mathcal{H})<\infty.$ Using this, they studied Birkhoff-James orthogonality on the space $(\mathcal{L}(\mathcal{H})^d,w_2(\cdot)).$ Recall that for $x,y\in \mathcal{X},$ $x$ is Birkhoff-James orthogonal \cite{B,J} to $y,$ if $\|x+\lambda y\|\geq \|x\|$ for all scalars $\lambda.$ 
 In \cite{M2}, the author studied Birkhoff-James orthogonality on $\mathcal{Z},$ for the case $d=1.$ Extreme point methods were used for this. Suppose $E_{\mathcal{X}}$ denotes the set of all extreme points of $B_{\mathcal{X}}.$
 In this article, we continue the study of \cite{GS,M2} in a more general setting and also improve some results significantly. Subdifferential sets of many norm functions have  been explored by several mathematicians. For example see \cite{GS,G,W,Z}. Here following a different approach, namely by finding the extreme points of $B_{\mathcal{Z}^*},$ we derive the subdifferential set of the joint numerical radius of $\mathcal{T}\in \mathcal{Z}.$ This shows the usefulness of extreme point methods in the study.  \\
 
 In Section 2, we first obtain the structure of $E_{\mathcal{Z^*}}$ for $d\geq 1.$ It extends and improves several previously known results, for example \cite[Th. 2.1 \& Th. 2.3]{M2} and \cite[Th. 4.6]{DMP}.   Using this, in Theorem \ref{th-03}, we obtain an expression for $J_{w_p}(\mathcal{T}),$ where $\mathcal{T}\in \mathcal{Z}.$ As a corollary, we derive $G_+(\mathcal{T},\mathcal{S})$ and $G_-(\mathcal{T},\mathcal{S}),$ where $\mathcal{T},\mathcal{S}\in \mathcal{Z}.$ We show that \cite[Th. 1.1]{GS} follows as a special case of Theorem \ref{th-03}. In Corollary \ref{cor-09}, we characterize Birkhoff-James orthogonality in $\mathcal{Z}$ using the expression for $J_{w_p}(\mathcal{T}).$ We show that \cite[Th. 1.2 \& Th. 1.3]{GS} and \cite[Th. 3.1]{M2} are particular cases of Corollary \ref{cor-09}. Finally in Section 4, we characterize the smooth points of $\mathcal{Z}$ and obtain an expression for $G(\mathcal{T},\mathcal{S}),$ if $\mathcal{T}$ is smooth. 

\section{Extreme points of the unit ball of $(\mathcal{L}(\mathcal{X})^d,w_p(\cdot))^*$}
The extreme points in dual of an operator space are of independent interest to several mathematicians. This is a crucial concept in the study of geometry of operators. In \cite{LO,RS} and in \cite{M2} the authors explored this area when the operator space is endowed with the usual operator norm and the numerical radius norm, respectively.  In this section, we explore this problem on the space of tuples of operators equipped with the joint numerical radius norm. We begin with the following two lemmas. The first lemma is an easy consequence of the Hahn-Banach theorem. For the reader's convenience, we include a proof here. Henceforth we assume $\frac{1}{p}+\frac{1}{q}=1.$
\begin{lemma}\label{lem-01}
Let $x=(x_1,\ldots,x_d)\in \ell_p^d.$ Then $$\|x\|_p=\sup\bigg\{\Big|\sum_{i=1}^d\alpha_ix_i\Big|:(\alpha_1,\ldots,\alpha_d)\in B_{\ell_q^d}\bigg\}.$$	
\end{lemma}
\begin{proof}
Suppose $y=(\alpha_1,\ldots,\alpha_d)\in B_{\ell_q^d}.$ Then by Holder's inequality, \[\Big|\sum_{i=1}^d\alpha_ix_i\Big|\leq \sum_{i=1}^d|\alpha_i||x_i|\leq \|y\|_q\|x\|_p\leq \|x\|_p,\]
which implies that	$\sup\bigg\{\Big|\sum_{i=1}^d\alpha_ix_i\Big|:(\alpha_1,\ldots,\alpha_d)\in B_{\ell_q^d}\bigg\}\leq \|x\|_p.$ On the other hand, choosing $\alpha_i=\frac{\overline{x_i}|x_i|^{p-2}}{\|x\|_p^{p-1}}$ for $1\leq i\leq d,$ we get $(\alpha_1,\ldots,\alpha_d)\in B_{\ell_q^d}$ and $\Big|\sum_{i=1}^d\alpha_ix_i\Big|=\|x\|_p.$ This proves the reverse inequality.
\end{proof}
In the next lemma, we obtain an alternative expression for the joint numerical radius. In what follows the notation $(\alpha_1x^*\otimes x,\ldots,\alpha_dx^*\otimes x)$ denotes the linear functional on $\mathcal{L}(\mathcal{X})^d$ defined as $$(\alpha_1x^*\otimes x,\ldots,\alpha_dx^*\otimes x)(T_1,\ldots,T_d):=\sum_{i=1}^d\alpha_i(x^*(T_ix)),$$ where $\alpha_i \in \mathbb{F}$ for $1\leq i\leq d,$ $x\in \mathcal{X}, x^*\in \mathcal{X}^*$  and $(T_1,\ldots,T_d)\in \mathcal{L}(\mathcal{X})^d.$ 
\begin{lemma}
	Let $\mathcal{T}=(T_1,\ldots,T_d)\in \mathcal{Z}.$ Then 
\[w_p(\mathcal{T})=\sup\bigg\{\Big|\sum_{i=1}^d\alpha_i(x^*(T_ix))\Big|:(\alpha_1x^*\otimes x,\ldots,\alpha_dx^*\otimes x)\in \mathcal{A}_d\bigg\},\] where 
\begin{equation}\label{eq-02}
\mathcal{A}_d=\{(\alpha_1x^*\otimes x,\ldots,\alpha_dx^*\otimes x):(\alpha_1,\ldots,\alpha_d)\in B_{\ell_q^d}, (x,x^*)\in \Pi\}.
\end{equation}
\end{lemma}
\begin{proof}
Let 	$(\alpha_1x^*\otimes x,\ldots,\alpha_dx^*\otimes x)\in \mathcal{A}_d.$ Suppose $y=(\alpha_1,\ldots,\alpha_d)$ and $z=(x^*(T_1x),\ldots,x^*(T_dx)).$ Then 
 \[\Big|\sum_{i=1}^d\alpha_i(x^*(T_ix))\Big|\leq \sum_{i=1}^d|\alpha_i||x^*(T_ix)|\leq \|y\|_q\|z\|_p\leq \|z\|_p\leq w_p(\mathcal{T}),\]
 which implies that $$\sup\bigg\{\Big|\sum_{i=1}^d\alpha_i(x^*(T_ix))\Big|:(\alpha_1x^*\otimes x,\ldots,\alpha_dx^*\otimes x)\in \mathcal{A}_d\bigg\}\leq w_p(\mathcal{T}).$$
 For the reverse inequality, observe that for $(z,z^*)\in \Pi,$ 
 \begin{eqnarray*}
 	\Big(\sum_{i=1}^d|z^*(T_iz)|^p\Big)^{\frac{1}{p}}&=&\|(z^*(T_1z),\ldots,z^*(T_dz))\|_p\\
 	&=&\sup\bigg\{\Big|\sum_{i=1}^d\alpha_iz^*(T_iz)\Big|:(\alpha_1,\ldots,\alpha_d)\in B_{\ell_q^d}\bigg\}, ~(\text{by Lemma \ref{lem-01}})\\
 	&\leq&\sup\bigg\{\Big|\sum_{i=1}^d\alpha_i(x^*(T_ix))\Big|:(\alpha_1x^*\otimes x,\ldots,\alpha_dx^*\otimes x)\in \mathcal{A}_d\bigg\}\\
 	\Rightarrow w_p(\mathcal{T})&\leq& \sup\bigg\{\Big|\sum_{i=1}^d\alpha_i(x^*(T_ix))\Big|:(\alpha_1x^*\otimes x,\ldots,\alpha_dx^*\otimes x)\in \mathcal{A}_d\bigg\}.
 	\end{eqnarray*}
\end{proof}

 In the following theorem, we study the extreme points of $B_{\mathcal{Z}^*}.$ We follow the idea of \cite[Th. 2.1]{M2}. Here we need the notion of polar and prepolar of sets. For $A\subseteq \mathcal{X},$ the polar of $A,$ denoted as $A^o,$ is defined as $$A^o=\{x^*\in \mathcal{X}^*:|x^*(a)|\leq 1 ~\forall~a\in A\},$$
  whereas for $B\subseteq \mathcal{X}^*,$ the prepolar of $B,$ denoted as ${}^oB,$ is defined as   \[{}^oB=\{x\in \mathcal{X}:|b^*(x)|\leq 1 ~\forall~b^*\in B\}.\] Moreover, $co(A)$ denotes the convex hull of $A.$
 \begin{theorem}\label{th-01}
 %	Let $\mathcal{Z}=(\mathcal{L}(\mathcal{X})^d,w_p(\cdot)).$  
 Suppose $\mathcal{A}_d$ is as in (\ref{eq-02}). Then $E_{\mathcal{Z}^*}\subseteq  {\overline{\mathcal{A}_d}}^{w*}.$
 \end{theorem}
 \begin{proof}
 	We first claim that ${}^o\mathcal{A}_d=B_{\mathcal{Z}}.$ Indeed
 	\begin{eqnarray*}
 		{}^o\mathcal{A}_d&=& \Big\{\mathcal{T}\in \mathcal{Z}:|(\alpha_1x^*\otimes x,\ldots,\alpha_dx^*\otimes x)(\mathcal{T})|\leq 1 ,\\
 		&&\quad\quad\quad \quad\quad\quad~\forall~ (\alpha_1x^*\otimes x,\ldots,\alpha_dx^*\otimes x)\in \mathcal{A}_d\Big\}\\
 		&=& \Big\{\mathcal{T}=(T_1,\ldots,T_d)\in \mathcal{Z}:\sup_{(\alpha_1x^*\otimes x,\ldots,\alpha_dx^*\otimes x)\in \mathcal{A}_d}|\sum_{i=1}^d\alpha_i(x^*(T_ix))|\leq 1 \Big\}\\
 		&=& \Big\{\mathcal{T}\in \mathcal{Z}:w_p(\mathcal{T})\leq 1 \Big\}\\
 		&=& B_{\mathcal{Z}}.
 	\end{eqnarray*}
 Now observe that
 \begin{eqnarray*}
 	({}^o\mathcal{A}_d)^o&=&(B_{\mathcal{Z}})^o\\
 	&=&\{f\in {\mathcal{Z}}^*:|f(T)|\leq 1~\forall~T\in B_{\mathcal{Z}}\}\\
 	&=&\{f\in {\mathcal{Z}}^*:\|f\|\leq 1\}\\
 	&=&B_{\mathcal{Z}^*}.
 \end{eqnarray*} 
Note that $\mathcal{A}_d$ is a Balanced set, i.e., for any scalar $\lambda$ with $|\lambda|\leq1,$ $\lambda \mathcal{A}_d\subseteq \mathcal{A}_d.$ Therefore, using a consequence of the Bipolar theorem (see \cite[Cor. 1.9, pp. 127]{CONW90}), we get,
\begin{equation*}\label{eq-bipolar}
	B_{\mathcal{Z}^*}=({}^o\mathcal{A}_d)^o=\overline{co(\mathcal{A}_d)}^{w*}.	
\end{equation*}
Clearly, $B_{\mathcal{Z}^*}$ is a weak*compact, convex subset of $\mathcal{Z}^*$ and $\mathcal{A}_d\subseteq B_{\mathcal{Z}^*}.$ Now, it follows from \cite[Th. 7.8, pp. 143]{CONW90}, that $E_{\mathcal{Z}^*}\subseteq {\overline{\mathcal{A}_d}}^{w*}.$
 	\end{proof}
 
 Next we prove the main theorem of this section, which is also the primary tool of this article. We show that the extreme points of $B_{\mathcal{Z}^*}$ assumes a nice form, if $\mathcal{X}$ is finite-dimensional. 
 \begin{theorem}\label{th-02}
 		Let $\dim(\mathcal{X})<\infty.$ Then $E_{\mathcal{Z}^*}\subseteq \mathcal{U}_d,$ where
 	\begin{equation*}
 	\mathcal{U}_d=	\{(\alpha_1x^*\otimes x,\ldots,\alpha_dx^*\otimes x):(\alpha_1,\ldots,\alpha_d)\in S_{\ell_q^d},x^*\in E_{\mathcal{X}^*},x\in E_{\mathcal{X}},x^*(x)=1\}.
 	\end{equation*} 
 \end{theorem}
 \begin{proof}
 	Observe that 
 	\begin{eqnarray*}
 	\mathcal{A}_1&=&\{\alpha x^*\otimes x:|\alpha|\leq 1,(x,x^*)\in \Pi\}\\
 	&=&\{x^*\otimes x:x^*\in B_{\mathcal{X}^*},x\in B_{\mathcal{X}},|x^*(x)|=\|x^*\|\|x\|\}.
 	\end{eqnarray*}
 From the proof of \cite[Th. 2.3]{M2} it follows that $\mathcal{A}_1$ is a compact set.  Therefore, the cartesian product $\underbrace{\mathcal{A}_1\times\ldots\times\mathcal{A}_1}_{d}$ is a compact set in the corresponding product topology. Clearly, $\mathcal{A}_d\subseteq \underbrace{\mathcal{A}_1\times\ldots\times\mathcal{A}_1}_{d}.$ It is easy to check that $\mathcal{A}_d$ is a closed subset of $ \underbrace{\mathcal{A}_1\times\ldots\times\mathcal{A}_1}_{d}$ and hence $\mathcal{A}_d$ is compact. Now by Theorem \ref{th-01}, 
 $$E_{\mathcal{Z}^*}\subseteq {\overline{\mathcal{A}_d}}^{w*}=\overline{\mathcal{A}_d}=\mathcal{A}_d.$$
Thus, if $f\in E_{\mathcal{Z}^*},$ then $f=(\alpha_1x^*\otimes x,\ldots, \alpha_dx^*\otimes x)$ for some $(x,x^*)\in \Pi$ and  $y=(\alpha_1,\ldots,\alpha_d)\in B_{\ell_q^d}.$  To complete the proof it only remains to show that $y\in S_{\ell_q^d},$ $x^*\in E_{{\mathcal{X}^*}}$ and $x\in E_{\mathcal{X}}.$ If possible, suppose that $x\notin E_{\mathcal{X}}.$ Then there exist $t\in (0,1), u,v\in B_{\mathcal{X}}$ such that $u\neq x\neq v$ and $x=tu+(1-t)v.$ Thus, from
 $$1=x^*(x)=tx^*(u)+(1-t)x^*(v)\leq t|x^*(u)|+(1-t)|x^*(v)|\leq 1,$$
 we get $x^*(u)=x^*(v)=1,$ that is $(u,x^*),(v,x^*)\in \Pi.$ Therefore, $$(\alpha_1x^*\otimes u,\ldots, \alpha_dx^*\otimes u), ~(\alpha_1x^*\otimes v,\ldots, \alpha_dx^*\otimes v)\in \mathcal{A}_d\subseteq B_{\mathcal{Z}^*}.$$ Moreover, 
 \[(\alpha_1x^*\otimes x,\ldots, \alpha_dx^*\otimes x)=t(\alpha_1x^*\otimes u,\ldots, \alpha_dx^*\otimes u)+(1-t)(\alpha_1x^*\otimes v,\ldots, \alpha_dx^*\otimes v).\] Now from $f=(\alpha_1x^*\otimes x,\ldots, \alpha_dx^*\otimes x)\in E_{\mathcal{Z}^*},$ it follows that $$x^*\otimes x=x^*\otimes u=x^*\otimes v.$$ Therefore, by \cite[Lem. 2.2]{M2},  we have $x=u=v,$ which is a contradiction and so $x\in E_{\mathcal{X}}.$ Similarly, it can be shown that $x^*\in E_{{\mathcal{X}}^*}.$ On the other hand, if $y\notin S_{\ell_q^d},$ then $\|y\|_q<1.$ Since $\|f\|=1,$ $y\neq 0,$ that is $0<\|y\|_q<1.$ Choose $f_1=\Big(\frac{\alpha_1}{\|y\|_q}x^*\otimes x,\ldots, \frac{\alpha_d}{\|y\|_q}x^*\otimes x\Big)$ and $f_2=(0x^*\otimes x,\ldots, 0x^*\otimes x).$ Clearly $f_1,f_2\in \mathcal{A}_d\subseteq B_{\mathcal{Z}^*}$ and
 \[f=(\alpha_1x^*\otimes x,\ldots, \alpha_dx^*\otimes x)=\|y\|_qf_1+(1-\|y\|_q)f_2.\]
 Since $f\in E_{\mathcal{Z}^*},$ so we must have $f=f_1=f_2,$ which implies that $y_i=0$ for $1\leq i\leq d,$ that is $y=0.$ This contradiction proves that $y\in S_{\ell_q^d}.$
 \end{proof}
 We would like to remark here that Theorem \ref{th-02} generalizes \cite[Th. 2.3]{M2}. As a consequence of the previous theorem, we get the following corollary.
 \begin{cor}\label{cor-01}
 		Let $\dim(\mathcal{X})<\infty.$  Then $B_{\mathcal{Z}^*}=co( \mathcal{U}_d),$ where $\mathcal{U}_d$ is as in Theorem \ref{th-02}.
 \end{cor}
 \begin{proof}
 	Since $\mathcal{X}$ is finite-dimensional,
 	$$B_{\mathcal{Z}^*}=co(E_{\mathcal{Z}^*})\subseteq co(\mathcal{U}_d)\subseteq B_{\mathcal{Z}^*},$$
 	consequently, $B_{\mathcal{Z}^*}=co(\mathcal{U}_d).$ 
 \end{proof}

The reverse inclusion in Theorem \ref{th-02}  also holds for a large class of Banach spaces. Recall that a Banach space $\mathcal{X}$ is said to be strictly convex, if $E_{\mathcal{X}}=S_{\mathcal{X}}.$ Moreover, $x\in S_{\mathcal{X}}$ is an exposed point of $\mathcal{X}$ if there exists a hyperplane $H$ such that $S_{\mathcal{X}}\cap H=\{x\}.$ In particular, if $\mathcal{X}$ is strictly convex, then each unit vector is an exposed point.

\begin{theorem}\label{th-05}
		Let $\dim(\mathcal{X})<\infty.$ Suppose $(x,x^*)\in \Pi$ such that $x\in E_{\mathcal{X}}$ and $x^*\in E_{\mathcal{X}^*}.$ Suppose either of the following holds.\\
		\rm(i) Either $x$ or $x^*$ is smooth.\\
		\rm(ii) $x$ is an exposed point. \\
		Then for each $(\alpha_1,\ldots,\alpha_d)\in S_{\ell_q^d},$  $(\alpha_1x^*\otimes x,\ldots,\alpha_dx^*\otimes x)\in E_{\mathcal{Z}^*}.$
	\end{theorem}
\begin{proof}
(i)	 Define $T:\mathcal{X}\to \mathcal{X}$ by 
	\[T(z)=x^*(z) x, ~\forall~z\in \mathcal{X}.\]
Since for $(z,z^*)\in \Pi,$ $|z^*(Tz)|=|x^*(z)||z^*(x)|\leq 1,$ and $x^*(Tx)=1,$ so $w(T)=1.$
% Moreover, if $(z,z^*)\in J_w(T),$ then $|x^*(z)|=|z^*(x)|=1.$\\
Let $\mathcal{T}=(T_1,\ldots,T_d)\in \mathcal{Z},$ where $T_i=\overline{\alpha_i}|\alpha_i|^{q-2}T, ~1\leq i\leq d.$ Then 
\begin{eqnarray*}
w_p(\mathcal{T})&=&\sup\Big\{\Big(\sum_{i=1}^d|z^*(T_iz)|^p\Big)^\frac{1}{p}:(z,z^*)\in \Pi\Big\}\\	
&=& \sup\Big\{\Big(\sum_{i=1}^d|\alpha_i|^q|z^*(Tz)|^p\Big)^\frac{1}{p}:(z,z^*)\in \Pi\Big\}\\
%&\leq&\Big(\sum_{i=1}^d|\alpha_i|^q\Big)^\frac{1}{p}=1,\\
&=& \sup\Big\{|z^*(Tz)|:(z,z^*)\in \Pi\Big\}=w(T)=1.
\end{eqnarray*}
%where equality holds if $|z^*(Tz)|=1.$
Clearly, $f=\Big(\alpha_1x^*\otimes x,\ldots,\alpha_dx^*\otimes x \Big)\in J_{w_p}(\mathcal{T}).$ If possible, assume that $f\notin E_{\mathcal{Z}^*}.$ Then by Corollary \ref{cor-01}, $f=\sum_{j=1}^mt_jf_j,$ where $(f\neq ) f_j\in \mathcal{U}_d,$ $t_j>0,$ and $t_1+\ldots+t_m=1.$ Now from
\[1=f(\mathcal{T})=\sum_{j=1}^mt_jf_j(\mathcal{T})\leq \sum_{j=1}^mt_j=1,\]
it follows that $f_j(\mathcal{T})=1,$ that is $f_j\in J_{w_p}(\mathcal{T})$ for all $1\leq j\leq m.$  Since $f_j\in \mathcal{U}_d,$ assume that $f_j=\bigg(\alpha_{1j}x_j^*\otimes x_j,\ldots,\alpha_{dj}x_j^*\otimes x_j\bigg),$ where $y_j=(\alpha_{1j},\ldots,\alpha_{dj})\in S_{\ell_q^d},$  $x_j^*\in E_{\mathcal{X}^*},  x_j\in E_{\mathcal{X}}$ and $x_j^*(x_j)=1.$ Therefore, for all $1\leq j\leq m,$
\begin{eqnarray*}
1&=&f_j(\mathcal{T})	\\
&=&\sum_{i=1}^d\alpha_{ij}x_j^*(T_ix_j)\\
&=&x_j^*(Tx_j)\Big(\sum_{i=1}^d\alpha_{ij}\overline{\alpha_i}|\alpha_i|^{q-2}\Big)\\
&\leq&|x_j^*(Tx_j)|\Big(\sum_{i=1}^d|\alpha_{ij}||\alpha_i|^{q-1}\Big)\\
&\leq & \|y_j\|_q\|(|\alpha_1|^{q-1},\ldots,|\alpha_d|^{q-1})\|_p\\
&=&1.
\end{eqnarray*}
This implies that $x_j^*(Tx_j)=1$ and $\alpha_{ij}=\alpha_i.$
Now, 
\begin{equation}\label{eq-04}
1=x_j^*(Tx_j)=x^*(x_j)x_j^*(x)\quad \Rightarrow \quad |x_j^*(x)|=|x^*(x_j)|=1.
\end{equation}
We consider the following two cases separately.
\begin{center}
Case A: $x$ is smooth, $\quad$ Case B: $x^*$ is smooth.
\end{center}
Case A:
Assume that $x$ is smooth. Then $J(x)=\{x^*\}.$ Using (\ref{eq-04}), if $\lambda_jx_j^*(x)=1,$ for some scalar $\lambda_j$ with $|\lambda_j|=1,$ then $\lambda_jx_j^*\in J(x),$ and so $\lambda_jx_j^*=x^*,$ that is, $x_j^*=\overline{\lambda_j}x^*.$ Therefore,
\begin{eqnarray*}
	f&=&\sum_{j=1}^mt_jf_j\\
	\Rightarrow \Big(\alpha_1x^*\otimes x,\ldots,\alpha_dx^*\otimes x \Big)&=&\sum_{j=1}^mt_j\Big(\alpha_1\overline{\lambda_j}x^*\otimes x_j,\ldots,\alpha_d\overline{\lambda_j}x^*\otimes x_j \Big)\\
	\Rightarrow \alpha_ix^*\otimes x&=&\sum_{j=1}^mt_j\alpha_i\overline{\lambda_j}x^*\otimes x_j ,\quad \forall~ 1\leq i\leq d \\
	\Rightarrow \alpha_ix^*\otimes (x-\sum_{j=1}^m t_j\overline{\lambda_j}x_j)&=&0\\
	\Rightarrow x&=&\sum_{j=1}^mt_j \overline{\lambda_j}x_j, \quad (\text{since } \alpha_i\neq 0, \text{ for some } i)\\
	\Rightarrow x&=& \overline{\lambda_j}x_j, \quad \forall~1\leq j\leq m,\quad(\text{since } x\in E_{\mathcal{X}})\\
	\Rightarrow x_j^*\otimes x_j&=& \overline{\lambda_j}x^*\otimes \lambda_jx=x^*\otimes x\\
	\Rightarrow f_j&=&f, \quad \forall\quad 1\leq j\leq m.  
	\end{eqnarray*}
This contradiction proves that $f\in E_{\mathcal{Z}^*}.$\\
Case B: Suppose $x^*$ is smooth. Then from (\ref{eq-04}) and \cite[Ch. 6]{MPS}, it follows that $x_j=\mu_jx$ for some scalar $\mu_j$ with $|\mu_j|=1.$ Now, proceeding similarly as Case A, it is easy to check that $f=f_j$ for all $1\leq j\leq m.$ This contradiction proves that $f\in E_{\mathcal{Z}^*}.$\\

(ii) Suppose $x$ is an exposed point and $H$ is a hyperplane such that $H\cap S_{\mathcal{X}}=\{x\}.$ Let $H_1$ be the hyperspace and $a\in \mathcal{X}$ such that $H=a+H_1.$ Consider $u^*\in S_{\mathcal{X}^*}$ defined as follows:
\[u^*(x)=1, \text{ and }u^*(h)=0, \text{ for all } h\in H_1.\] It is easy to check that $u^*\in S_{\mathcal{X}^*}$ and if $|u^*(y)|=1$  for some $y\in S_{\mathcal{X}},$ then $y=\lambda x$ for some scalar $\lambda$ with $|\lambda|=1.$ Thus, $u^*$ is smooth. Define $T:\mathcal{X}\to \mathcal{X}$ by 
\[T(z)=u^*(z) x, ~\forall~z\in \mathcal{X}.\]
Since for $(z,z^*)\in \Pi,$ $|z^*(Tz)|=|u^*(z)||z^*(x)|\leq 1,$ and $x^*(Tx)=1,$ so $w(T)=1.$
% Moreover, if $(z,z^*)\in J_w(T),$ then $|x^*(z)|=|z^*(x)|=1.$\\
Let $\mathcal{T}=(T_1,\ldots,T_d)\in \mathcal{Z},$ where $T_i=\overline{\alpha_i}|\alpha_i|^{q-2}T, ~1\leq i\leq d.$ Now we proceed  as (i). If possible, assume that  $f=\Big(\alpha_1x^*\otimes x,\ldots,\alpha_dx^*\otimes x \Big)\notin E_{\mathcal{Z}^*}.$ Then $f=\sum_{j=1}^mt_jf_j,$ where $(f\neq ) f_j\in \mathcal{U}_d,$ $t_j>0,$ and $t_1+\ldots+t_m=1.$ Similarly as (i), we get the following for $1\leq j\leq m.$
\begin{enumerate}
 \item $f_j\in J_{w_p}(\mathcal(T))$
  \item $f_j=(\alpha_1x_j^*\otimes x_j,\ldots,\alpha_d x_j^*\otimes x_j)$ for some $x_j^*\in E_{\mathcal{X}^*},x_j\in E_{\mathcal{X}}$ and $x_j^*(x_j)=1.$
  \item $x_j^*(Tx_j)=1\Rightarrow u^*(x_j)x_j^*(x)=1\Rightarrow |u^*(x_j)|=1.$
 \end{enumerate}
Therefore, $x_j=\lambda_jx$ for some scalar $\lambda_j$ with $|\lambda_j|=1.$ Now, similarly as (i) it can be proved that $f=f_j$ for all $1\leq j\leq m,$ and so $f\in E_{\mathcal{Z}^*}.$
\end{proof}
Combining Theorem \ref{th-02} and Theorem \ref{th-05}, we immediately get the following corollary. Recall that a finite-dimensional Banach space is said to be polyhedral, if its unit ball has finitely many extreme points.
A unit vector $x$ is said to be $k$-smooth, if $J(x)$ contains exactly $k$ linearly independent functionals.
\begin{cor}\label{cor-11}
		Let $\dim(\mathcal{X})<\infty.$  Suppose either of the following holds.
		\begin{center}
	\rm(i) $\mathcal{X}$ is smooth.
	\rm(ii) $\mathcal{X}^*$ is smooth. 
	\rm(iii) $\mathcal{X}$ is polyhedral.
	\rm(iv) $\dim(\mathcal{X})=2.$
	\end{center}
	Then $E_{\mathcal{Z}^*}= \mathcal{U}_d.$
\end{cor}
\begin{proof}
If either (i) or (ii) holds, then the result follows from Theorem \ref{th-02} and  part (i) of Theorem \ref{th-05}.\\
 Observe that in a polyhedral Banach space every extreme point is an exposed point. Now, (iii) follows from 	Theorem \ref{th-02} and  part (ii) of Theorem \ref{th-05}.\\
 (iv) Suppose $\dim(\mathcal{X})=2.$ Let $x\in S_{\mathcal{X}}.$ If $x$ is not smooth, then $J(x)$ has exactly two linearly independent functionals. Hence $x$ is $2$-smooth, and so by \cite[Th. 4.2]{Wo}, $x$ is an exposed point. Thus, every unit vector of $\mathcal{X}$ is either smooth or exposed point. Now, the result follows using Theorem \ref{th-02} and  Theorem \ref{th-05}.
\end{proof}

Note that, if $\mathcal{H}$ is a Hilbert space and $x^*(x)=1,$ where $x\in S_{\mathcal{H}},$ and $ x^*\in S_{\mathcal{H}^*},$ then by the Riesz-Representation theorem, $x^*(z)=\langle z,x\rangle$ for all $z\in \mathcal{H}.$ For this identification on a Hilbert space, we denote the functional $x^*\otimes x$ simply by $x\otimes x.$ To be precise, $x\otimes x(S)=\langle Sx,x\rangle$ for all  $S\in \mathcal{L}(\mathcal{H}).$ Now, since a Hilbert space is smooth, by Corollary \ref{cor-11}, we get the following.

\begin{cor}
Let $\dim(\mathcal{H})<\infty,$ and $\mathcal{Z}=(\mathcal{L}(\mathcal{H})^d,w_p(\cdot)).$ Then 	
\[E_{\mathcal{Z}^*}=\{(\alpha_1x\otimes x,\ldots,\alpha_dx\otimes x):(\alpha_1,\ldots,\alpha_d)\in S_{\ell_q^d},x\in S_{\mathcal{H}}\}.\]
In particular, if $d=1,$ then 
\[E_{\mathcal{Z}^*}=\{\alpha x\otimes x:|\alpha|=1,x\in S_{\mathcal{H}}\}.\]
\end{cor}

We would like to end this section with the following remark.
\begin{remark}
In \cite[Th. 4.6]{DMP}, the authors proved that $E_{\mathcal{Z}^*}=\mathcal{U}_d$ for $d=1,$ assuming $\mathcal{X}$ is a two-dimensional polyhedral Banach space. So Corollary \ref{cor-11}	improves \cite[Th. 4.6]{DMP} significantly. For $d=1,$ Corollary \ref{cor-11} also strengthens \cite[Th. 2.3]{M2} for a large class of Banach spaces.
\end{remark}

\section{Subdifferential set of $w_p(\cdot)$ at $\mathcal{T}$}
In this section, we deal with the subdifferential set of the joint numerical radius of a tuple. Henceforth, the symbol $M_{w_p(\mathcal{T})}$ denotes the joint  numerical radius attaining set of $\mathcal{T}=(T_1,\ldots,T_d),$ that is,
\[M_{w_p(\mathcal{T})}=\bigg\{(x,x^*)\in \Pi:\Big(\sum_{i=1}^d|x^*(T_ix)|^p\Big)^{\frac{1}{p}}=w_p(\mathcal{T})\bigg\}.\]
The main result of this section is as follows.
\begin{theorem}\label{th-03}
	Let $\dim(\mathcal{X})<\infty.$ Then for $\mathcal{T}=(T_1,\ldots,T_d)\in \mathcal{Z},$ 
	\[J_{w_p}(\mathcal{T})=co(\mathcal{B}), \text{ where}\] 
	  \begin{eqnarray*}
		\mathcal{B}&=&\bigg\{\frac{1}{w_p(\mathcal{T})^{p-1}}\bigg(\overline{x^*(T_1x)}|x^*(T_1x)|^{p-2}x^*\otimes x,\ldots,\overline{x^*(T_dx)}|x^*(T_dx)|^{p-2}x^*\otimes x\bigg):\\
		&&\quad \quad \quad \quad \quad \quad(x,x^*)\in M_{w_p(\mathcal{T})}, x\in E_{\mathcal{X}},x^*\in E_{\mathcal{X}^*}\bigg\}.
	\end{eqnarray*}
\end{theorem}
\begin{proof}
We first show that $co(\mathcal{B})\subseteq J_{w_p}(\mathcal{T}).$
Let $$f=\frac{1}{w_p(\mathcal{T})^{p-1}}\bigg(\overline{x^*(T_1x)}|x^*(T_1x)|^{p-2}x^*\otimes x,\ldots,\overline{x^*(T_dx)}|x^*(T_dx)|^{p-2}x^*\otimes x\bigg)\in \mathcal{B}.$$ Note that, $\frac{1}{w_p(\mathcal{T})^{p-1}}\bigg(\overline{x^*(T_1x)}|x^*(T_1x)|^{p-2},\ldots,\overline{x^*(T_dx)}|x^*(T_dx)|^{p-2}\bigg)\in S_{\ell_q^d},$ since $ (x,x^*)\in M_{w_p(\mathcal{T})}.$ Therefore $f\in \mathcal{U}_d\subseteq B_{\mathcal{Z}^*}.$ Now,
\begin{eqnarray*}
	f(\mathcal{T})&=&\frac{1}{w_p(\mathcal{T})^{p-1}}\sum_{i=1}^d\overline{x^*(T_ix)}|x^*(T_ix)|^{p-2}x^*\otimes x(T_i)\\
	&=&\frac{1}{w_p(\mathcal{T})^{p-1}}\sum_{i=1}^d|x^*(T_ix)|^{p}\\
	&=& w_p(\mathcal{T}),\quad \quad \quad (\text{since } (x,x^*)\in M_{w_p(\mathcal{T})}),
	\end{eqnarray*}
which implies that $\|f\|=1$ and $f\in J_{w_p}(\mathcal{T}).$ Thus, $\mathcal{B}\subseteq J_{w_p}(\mathcal{T}).$ Since $J_{w_p}(\mathcal{T})$ is a convex set, so $co(\mathcal{B})\subseteq J_{w_p}(\mathcal{T}).$\\
On the other hand, suppose $f\in J_{w_p}(\mathcal{T}).$ Then by Corollary \ref{cor-01}, $f\in co(\mathcal{U}_d).$ So there exist $t_i>0,f_i\in \mathcal{U}_d$ for $1\leq i\leq m$ such that $t_1+\ldots+t_m=1$ and $f=t_1f_1+\ldots+t_mf_m.$ Since $\|f_i\|\leq 1,$ so $f_i(\mathcal{T})\leq w_p(\mathcal{T})$ for all $1\leq i\leq m.$ Now from
\[w_p(\mathcal{T})=f(\mathcal{T})=\sum_{i=1}^mt_if_i(\mathcal{T})\leq \sum_{i=1}^mt_iw_p(\mathcal{T})=w_p(\mathcal{T}),\]
it follows that $f_i(\mathcal{T})=w_p(\mathcal{T}),$ that is $f_i\in J_{w_p}(\mathcal{T})$ for all $1\leq i\leq m.$  Since $f_i\in \mathcal{U}_d,$ assume that $f_i=\bigg(\alpha_{1i}x_i^*\otimes x_i,\ldots,\alpha_{di}x_i^*\otimes x_i\bigg),$ where $y_i=(\alpha_{1i},\ldots,\alpha_{di})\in S_{\ell_q^d},$  $x_i^*\in E_{\mathcal{X}^*},  x_i\in E_{\mathcal{X}}$ and $x_i^*(x_i)=1.$ Now,
\[w_p(\mathcal{T})=f_i(\mathcal{T})=\sum_{j=1}^d\alpha_{ji}x_i^*(T_jx_i)\leq \|y_i\|_q\|(x_i^*(T_1x_i),\ldots,x_i^*(T_dx_i))\|_p\leq w_p(\mathcal{T}),\] and so from the equality condition, we have
\[w_p(\mathcal{T})=\|(x_i^*(T_1x_i),\ldots,x_i^*(T_dx_i))\|_p=\Big(\sum_{j=1}^d|x_i^*(T_jx_i)|^p\Big)^{\frac{1}{p}},\]
that is $(x_i,x_i^*)\in M_{w_p(\mathcal{T})}.$
Moreover,  for all $1\leq j\leq d,$
\[\alpha_{ji}=\frac{1}{w_p(\mathcal{T})^{p-1}}\overline{x_i^*(T_jx_i)}|x_i^*(T_jx_i)|^{p-2}.\]
Thus $f_i\in \mathcal{B},$ which proves that $f\in co(\mathcal{B}),$ and so $J_{w_p}(\mathcal{T})\subseteq co(\mathcal{B}).$
\end{proof}

Several corollaries of Theorem \ref{th-03} are in order now. As a straightforward consequence of Theorem \ref{th-03} and (\ref{eq-03}), we get the following expression for the right-hand and left-hand Gateaux derivatives of a tuple in the direction of another tuple.
\begin{cor}\label{cor-02}
	Let $\dim(\mathcal{X})<\infty.$ Then for $\mathcal{T}=(T_1,\ldots,T_d),\mathcal{S}=(S_1,\ldots,S_d) \in \mathcal{Z},$ 
		\begin{eqnarray*}
		G_+(\mathcal{T},\mathcal{S})
		&=&\frac{1}{w_p(\mathcal{T})^{p-1}}\max\{\mu:\mu\in C_p\}, \text{ and }\\
		G_-(\mathcal{T},\mathcal{S})
		&=&\frac{1}{w_p(\mathcal{T})^{p-1}}\min\{\mu:\mu\in C_p\}, \text{ where}
			\end{eqnarray*}
		\[C_p=\bigg\{\sum_{i=1}^d\Re\bigg(\overline{x^*(T_ix)}|x^*(T_ix)|^{p-2}x^*(S_ix) \bigg):(x,x^*)\in M_{w_p(\mathcal{T})}, x\in E_{\mathcal{X}},x^*\in E_{\mathcal{X}^*}\bigg\}.\]
\end{cor}
\begin{proof}
Suppose $G_+(\mathcal{T},\mathcal{S})=\Re\{f(\mathcal{S})\}$ for some $f\in J_{w_p}(\mathcal{T}).$	Then by Theorem \ref{th-03}, there exist $t_j>0$ and $f_j\in \mathcal{B}\subseteq J_{w_p}(\mathcal{T})$ for $1\leq j\leq n$ such that $t_1+\ldots+t_n=1$ and $f=\sum_{j=1}^nt_jf_j.$ By (\ref{eq-03}), $\Re\{f_j(\mathcal{S})\}\leq G_+(\mathcal{T},\mathcal{S}).$ Therefore,
\[G_+(\mathcal{T},\mathcal{S})=\Re\{f(\mathcal{S})\}=\Re\{\sum_{j=1}^nt_jf_j(\mathcal{S})\}\leq \sum_{j=1}^nt_jG_+(\mathcal{T},\mathcal{S})=G_+(\mathcal{T},\mathcal{S}),\] which shows that $G_+(\mathcal{T},\mathcal{S})=\Re\{f_j(\mathcal{S})\},$ for some $f_j\in \mathcal{B},$ and so
\begin{eqnarray*}
G_+(\mathcal{T},\mathcal{S})&=&\Re\{f_j(\mathcal{S})\}\\
&\leq& \max\{\Re\{g(\mathcal{S})\}:g\in \mathcal{B}\}\\
&\leq& \max\{\Re\{g(\mathcal{S})\}:g\in J_{w_p}(\mathcal{T})\}\\
&=&G_+(\mathcal{T},\mathcal{S})\\
\Rightarrow G_+(\mathcal{T},\mathcal{S})&=&  \max\{\Re\{g(\mathcal{S})\}:g\in \mathcal{B}\}\\
&=&\frac{1}{w_p(\mathcal{T})^{p-1}}\max\{\mu:\mu\in C_p\}.
\end{eqnarray*}
Similarly, we get the expression for $G_-(\mathcal{T},\mathcal{S}).$
\end{proof}

In particular, if $p=2,$ then Theorem \ref{th-03} and Corollary \ref{cor-02} assume the following simpler form.
\begin{cor}\label{cor-03}
		Let $\dim(\mathcal{X})<\infty.$ Then for $\mathcal{T}=(T_1,\ldots,T_d)\in \mathcal{Z},$ 
	\begin{eqnarray*}
		J_{w_2}(\mathcal{T})
		&=&co\bigg\{\frac{1}{w_2(\mathcal{T})}\bigg(\overline{x^*(T_1x)}x^*\otimes x,\ldots,\overline{x^*(T_dx)}x^*\otimes x\bigg):\\
		&&\quad \quad \quad \quad \quad \quad(x,x^*)\in M_{w_2(\mathcal{T})}, x\in E_{\mathcal{X}},x^*\in E_{\mathcal{X}^*}\bigg\}.
	\end{eqnarray*}
For each $\mathcal{S}=(S_1,\ldots,S_d) \in \mathcal{Z},$ 
	\begin{eqnarray*}
	G_+(\mathcal{T},\mathcal{S})
	&=&\frac{1}{w_2(\mathcal{T})}\max\{\mu:\mu\in C_2\}, \text{ and }\\
	G_-(\mathcal{T},\mathcal{S})
	&=&\frac{1}{w_2(\mathcal{T})}\min\{\mu:\mu\in C_2\}, \text{ where}
\end{eqnarray*}
\[C_2=\bigg\{\sum_{i=1}^d\Re\bigg(\overline{x^*(T_ix)}x^*(S_ix) \bigg):(x,x^*)\in M_{w_2(\mathcal{T})}, x\in E_{\mathcal{X}},x^*\in E_{\mathcal{X}^*}\bigg\}.\]
\end{cor}

Observe that if $d=1,$ then $w_p(\cdot)=w(\cdot)$ for all $1<p<\infty.$ Therefore, from Theorem \ref{th-03}, we get the following description of the  subdifferential set $J_w(T)=\{f\in S_{\mathcal{Z}^*}:f(T)=w(T)\}$ of the numerical radius of an operator $T$.

\begin{cor}\label{cor-10}
	Let $\dim(\mathcal{X})<\infty.$ Then for $T\in (\mathcal{L}(\mathcal{X}),w(\cdot)),$ 
	\begin{eqnarray*}
		J_{w}(T)
		=co\bigg\{\frac{1}{w(T)}\overline{x^*(Tx)}x^*\otimes x:(x,x^*)\in \Pi, x\in E_{\mathcal{X}},x^*\in E_{\mathcal{X}^*}, |x^*(Tx)|=w(T)\bigg\}.
	%	=co\bigg\{\frac{1}{w(T)}\overline{x^*(T_1x)}x^*\otimes x:(x,x^*)\in M_{w(\mathcal{T})}, x\in E_{\mathcal{X}},x^*\in E_{\mathcal{X}^*}\bigg\}.
	\end{eqnarray*}
\end{cor}

As an application of Theorem \ref{th-03}, we easily get the following characterization of Birkhoff-James orthogonality in $\mathcal{Z}.$

\begin{cor}\label{cor-05}
	Let $\dim(\mathcal{X})<\infty.$ Suppose  $\mathcal{V}$ is a subspace  of $\mathcal{Z},$ and $\mathcal{T}=(T_1,\ldots,T_d)\in \mathcal{Z}\setminus \mathcal{V}.$ Then the following are equivalent.\\
	\rm(i) $w_p(\mathcal{T}+\mathcal{S})\geq w_p(\mathcal{T})$  for all $\mathcal{S}=(S_1,\ldots,S_d)\in \mathcal{V}.$ \\
	\rm(ii) There exists $t_j>0, ~x_j\in E_{\mathcal{X}},~x_j^*\in E_{\mathcal{X}^*}$ for $1\leq j\leq n,$ such that  $t_1+\ldots+t_n=1,$ $(x_j,x_j^*)\in M_{w_p(\mathcal{T})}$ and  for all $\mathcal{S}=(S_1,\ldots,S_d)\in \mathcal{V},$
	\begin{eqnarray*}
	\sum_{j=1}^n\sum_{i=1}^dt_j\bigg(\overline{x_j^*(T_ix_j)}|x_j^*(T_ix_j)|^{p-2}x_j^*(S_ix_j) \bigg)=0.
	\end{eqnarray*}
\end{cor}
\begin{proof}
(i) $\Rightarrow$ (ii) Consider $\mathcal{M}=span\{\mathcal{T},\mathcal{V}\}.$ Define $f:\mathcal{M}\to \mathbb{F}$ by $f(\alpha \mathcal{T}+\mathcal{S})=\alpha w_p(\mathcal{T})$ for all scalars $\alpha$ and $\mathcal{S}\in \mathcal{V}.$	Now, 
\[|f(\alpha \mathcal{T}+\mathcal{S})|=|\alpha| w_p(\mathcal{T})=w_p(\alpha \mathcal{T})\leq w_p(\alpha \mathcal{T}+\mathcal{S})\] implies that $f\in \mathcal{M}^*$ and $\|f\|\leq 1.$ Moreover, $f(\mathcal{T})=w_p(\mathcal{T})$ implies that $\|f\|=1.$ So there exists $g\in \mathcal{Z}^*$ such that $g|_\mathcal{M}=f$ and $\|g\|=\|f\|=1.$ Therefore $g\in J_{w_p}(\mathcal{T})$ and $g(\mathcal{S})=0$ for all $\mathcal{S}\in \mathcal{V}.$ Now (ii) follows from Theorem \ref{th-03}.\\

(ii) $\Rightarrow $ (i) Without loss of generality, assume that $w_p(\mathcal{T})=1.$ For $1\leq j\leq n,$ let
\[f_j=\bigg(\overline{x_j^*(T_1x_j)}|x_j^*(T_1x_j)|^{p-2}x_j^*\otimes x_j,\ldots,\overline{x_j^*(T_dx_j)}|x_j^*(T_dx_j)|^{p-2}x_j^*\otimes x_j\bigg)\]
and $f=\sum_{j=1}^nt_jf_j.$ Then by Theorem \ref{th-03}, $f\in J_{w_p}(\mathcal{T}).$ Moreover, (ii) implies that $f(\mathcal{S})=0$ for all $\mathcal{S}\in \mathcal{V},$ and so
\[w_p(\mathcal{T}+\mathcal{S})\geq f(\mathcal{T}+\mathcal{S})=f(\mathcal{T})=w_p(\mathcal{T}).\]
\end{proof}

In particular, if $\mathcal{V}=\mathbb{F}^d\mathcal{S}:=\{\lambda \mathcal{S}=(\lambda_1S_1,\ldots, \lambda_dS_d):\lambda=(\lambda_1,\ldots,\lambda_d)\in \mathbb{F}^d\},$ then Corollary \ref{cor-05} assumes the following form. The proof is straightforward. So we skip the proof here. 

\begin{cor}\label{cor-09}
	Let $\dim(\mathcal{X})<\infty.$ Suppose $~\mathcal{T}=(T_1,\ldots,T_d), \mathcal{S}=(S_1,\ldots,S_d)\in \mathcal{Z}$ and $\mathcal{T}\notin \mathbb{F}^d\mathcal{S}.$  Then the following are equivalent.\\
	\rm(i) $w_p(\mathcal{T}+\lambda\mathcal{S})\geq w_p(\mathcal{T})$  for all $\lambda \in \mathbb{F}^d.$  \\
	\rm(ii) There exists $t_j>0, ~x_j\in E_{\mathcal{X}},~x_j^*\in E_{\mathcal{X}^*}$ for $1\leq j\leq n,$ such that  $t_1+\ldots+t_n=1,$ $(x_j,x_j^*)\in M_{w_p(\mathcal{T})}$ and  
	\begin{eqnarray*}
		\sum_{j=1}^nt_j\bigg(\overline{x_j^*(T_ix_j)}|x_j^*(T_ix_j)|^{p-2}x_j^*(S_ix_j) \bigg)=0, \quad \forall ~1\leq i\leq d.
	\end{eqnarray*}
	\end{cor}

For $p=2$ we have the following result, which is a  generalization of \cite[Th. 3.1]{M2}. %Observe that if $d=1,$ then $w_p(\cdot)=w(\cdot)$ for all $1<p<\infty.$ Therefore, the next corollary  is a generalization of \cite[Th. 3.1]{M2}.
\begin{cor}\label{cor-07}
	Let $\dim(\mathcal{X})<\infty.$ Suppose  $\mathcal{V}$ is a subspace  of $\mathcal{Z},$ and $\mathcal{T}=(T_1,\ldots,T_d)\in \mathcal{Z}\setminus \mathcal{V}.$  Then the following are equivalent.\\
	\rm(i) $w_2(\mathcal{T}+\mathcal{S})\geq w_2(\mathcal{T})$  for all $\mathcal{S}=(S_1,\ldots,S_d)\in \mathcal{V}.$ \\
	\rm(ii) There exists $t_j>0, ~x_j\in E_{\mathcal{X}},~x_j^*\in E_{\mathcal{X}^*}$ for $1\leq j\leq n,$ such that  $t_1+\ldots+t_n=1,$ $(x_j,x_j^*)\in M_{w_2(\mathcal{T})}$ and  for all $\mathcal{S}=(S_1,\ldots,S_d)\in \mathcal{V},$
	\begin{eqnarray*}
		\sum_{j=1}^n\sum_{i=1}^dt_j\bigg(\overline{x_j^*(T_ix_j)}x_j^*(S_ix_j) \bigg)=0.
	\end{eqnarray*}
\end{cor}

%Note that, if $\mathcal{H}$ is a Hilbert space and $x^*(x)=1,$ where $x\in S_{\mathcal{H}},$ and $ x^*\in S_{\mathcal{H}^*},$ then by the Riesz-Representation theorem, $x^*(z)=\langle z,x\rangle$ for all $z\in \mathcal{H}.$ For this identification on a Hilbert space, we denote the functional $x^*\otimes x$ simply by $x\otimes x.$ To be precise, $x\otimes x(S)=\langle Sx,x\rangle$ for all  $S\in \mathcal{L}(\mathcal{H}).$ 
Now in the setting of Hilbert space, Corollary \ref{cor-03} and Corollary \ref{cor-07} assume the following form, respectively.
\begin{cor}\label{cor-04}\cite[Th. 1.1]{GS}
	Let $\mathcal{H}$ be a Hilbert space, where $\dim(\mathcal{H})<\infty,$ and $\mathcal{Z}=(\mathcal{L}(\mathcal{H})^d,w_2(\cdot)).$ Then for $\mathcal{T}=(T_1,\ldots,T_d)\in \mathcal{Z},$ 
	\begin{eqnarray*}
		J_{w_2}(\mathcal{T})
		&=&co\bigg\{\frac{1}{w_2(\mathcal{T})}\bigg(\overline{\langle T_1x,x\rangle}x\otimes x,\ldots,\overline{\langle T_dx,x\rangle }x\otimes x\bigg): \\
		&& \quad \quad \quad \quad \|x\|=1, \sqrt{\sum_{i=1}^d|\langle T_ix,x\rangle|^2}=w_2(\mathcal{T})\bigg\}.
	\end{eqnarray*}
	For each $\mathcal{S}=(S_1,\ldots,S_d) \in \mathcal{Z},$ 
	\begin{eqnarray*}
	G_+(\mathcal{T},\mathcal{S})
	&=&\frac{1}{w_2(\mathcal{T})}\max\{\mu:\mu\in C_2\}, \text{ and }\\
	G_-(\mathcal{T},\mathcal{S})
	&=&\frac{1}{w_2(\mathcal{T})}\min\{\mu:\mu\in C_2\}, \text{ where}
\end{eqnarray*}
\[C_2=\bigg\{\sum_{i=1}^d\Re\bigg(\overline{\langle T_ix,x\rangle}\langle S_ix,x\rangle \bigg): \|x\|=1, \sqrt{\sum_{i=1}^d|\langle T_ix,x\rangle|^2}=w_2(\mathcal{T})\bigg\}.\]
\end{cor}

Note that \cite[Th. 1.2]{GS} is a particular case of the following corollary.
\begin{cor}\label{cor-08}
	Let $\dim(\mathcal{H})<\infty$ and $\mathcal{Z}=(\mathcal{L}(\mathcal{H})^d,w_2(\cdot)).$ Suppose $\mathcal{V}$ is a subspace  of $\mathcal{Z}$ and $\mathcal{T}=(T_1,\ldots,T_d)\in \mathcal{Z}\setminus \mathcal{V}.$ Then the following are equivalent.\\
	\rm(i) $w_2(\mathcal{T}+\mathcal{S})\geq w_2(\mathcal{T})$  for all $\mathcal{S}=(S_1,\ldots,S_d)\in \mathcal{V}.$ \\
	\rm(ii) There exists $t_j>0, ~x_j\in S_{\mathcal{H}}$ for $1\leq j\leq n,$ such that  $t_1+\ldots+t_n=1,$ $ \sqrt{\sum_{i=1}^d|\langle T_ix,x\rangle|^2}=w_2(\mathcal{T})$ and  for all $\mathcal{S}=(S_1,\ldots,S_d)\in \mathcal{V},$
	\begin{eqnarray*}
		\sum_{j=1}^n\sum_{i=1}^dt_j\bigg(\overline{\langle T_ix_j, x_j\rangle }\langle S_ix_j, x_j\rangle \bigg)=0.
	\end{eqnarray*}
\end{cor}
Similarly, \cite[Th. 1.3]{GS} is a particular case of Corollary \ref{cor-09} in the Hilbert space setting and for $p=2.$ We would like to end the section with the remark that the distance of a tuple from a subspace of $\mathcal{Z}$ can now be obtained from Theorem \ref{th-02}  (or Theorem \ref{th-03}) and  \cite[Th. 2.1]{MP22} proceeding similarly as \cite[Th. 3.5]{M2}.

\section{Gateaux derivative of a tuple}
In this section, using Theorem \ref{th-03} we first characterize the smooth tuples in $\mathcal{Z},$ and then derive the Gateaux derivative of a tuple.
\begin{theorem}\label{th-04}
		Let $\dim(\mathcal{X})<\infty.$  Suppose $\mathcal{T}=(T_1,\ldots,T_d)\in \mathcal{Z}$ and $w_p(\mathcal{T})\neq 0.$ Then the following are equivalent.\\
		\rm(i) The numerical radius $w_p(\cdot)$ of $\mathcal{Z}$ is Gateaux differentiable at $\mathcal{T}.$\\
		\rm(ii) There exists $(x_0,x_0^*)\in \Pi$ such that $x_0\in E_{\mathcal{X}}, x_0^*\in E_{\mathcal{X}^*}$ and $$M_{w_p(\mathcal{T})}=\{(\mu x_0, \overline{\mu} x_0^*):\mu\in \mathbb{F},|\mu|=1\}.$$
		Moreover, in this case, for each $\mathcal{S}=(S_1,\ldots,S_d) \in \mathcal{Z},$
	\begin{eqnarray*}
		G(\mathcal{T},\mathcal{S})
		&=&\frac{1}{w_p(\mathcal{T})^{p-1}}\sum_{i=1}^d\Re\bigg(\overline{x_0^*(T_ix_0)}|x_0^*(T_ix_0)|^{p-2}x_0^*(S_ix_0) \bigg).
	\end{eqnarray*}
\end{theorem}
\begin{proof}
	Without loss of generality, assume that $w_p(\mathcal{T})=1.$ \\
(i) $\Rightarrow $ (ii). Since $w_p(\cdot)$ is Gateaux differentiable at $\mathcal{T},$ $J_{w_p}(\mathcal{T})$ is singleton. Suppose $$J_{w_p}(\mathcal{T})=\bigg\{\bigg(\overline{x^*(T_1x)}|x^*(T_1x)|^{p-2}x^*\otimes x,\ldots,\overline{x^*(T_dx)}|x^*(T_dx)|^{p-2}x^*\otimes x\bigg)\bigg\}$$
for some $(x,x^*)\in M_{w_p(\mathcal{T})},$ where $x\in E_{\mathcal{X}}, x^*\in E_{\mathcal{X}^*}.$ We show that if
 $(y,y^*)\in M_{w_p(\mathcal{T})},$ then $(y,y^*)=(\mu x,\overline{\mu}x^*)$ for some scalar $\mu$ with $|\mu|=1.$ Clearly,
 \[\bigg(\overline{y^*(T_1y)}|y^*(T_1y)|^{p-2}y^*\otimes y,\ldots,\overline{y^*(T_dy)}|y^*(T_dy)|^{p-2}y^*\otimes y\bigg)\in J_{w_p}(\mathcal{T}).\]  
 Therefore for all $1\leq i\leq d$, 
 \begin{equation}\label{eq-01}
 \overline{y^*(T_iy)}|y^*(T_iy)|^{p-2}y^*\otimes y=\overline{x^*(T_ix)}|x^*(T_ix)|^{p-2}x^*\otimes x.
\end{equation}
 If possible, suppose that $\{x,y\}$ is linearly independent. Then there exists $S\in \mathcal{L}(\mathcal{X})$ such that $Sx=0$ and $Sy=y.$ Now from (\ref{eq-01}),
 \begin{eqnarray*}
 	 \overline{y^*(T_iy)}|y^*(T_iy)|^{p-2}y^*\otimes y(S)&=&\overline{x^*(T_ix)}|x^*(T_ix)|^{p-2}x^*\otimes x(S)\\
 	 \Rightarrow \overline{y^*(T_iy)}|y^*(T_iy)|^{p-2}&=&0, ~(\text{since } y^*(y)=1)\\
 	 \Rightarrow y^*(T_iy)&=&0,
 	\end{eqnarray*}
 which contradicts that $(y,y^*)\in M_{w_p(\mathcal{T})}.$ Therefore, assume that $y=\mu x.$ Since $y,x\in S_{\mathcal{X}},$ so $|\mu|=1.$ Now if possible, assume that $\{x^*,y^*\}$ is linearly independent. Choose $u\in \ker(x^*)\setminus \ker(y^*).$ Consider $S\in \mathcal{L}(\mathcal{X})$ defined by $S(z)=x^*(z)u$ for all $z\in \mathcal{X}.$ Again from (\ref{eq-01}), 
  \begin{eqnarray*}
 	\overline{y^*(T_iy)}|y^*(T_iy)|^{p-2}y^*\otimes y(S)&=&\overline{x^*(T_ix)}|x^*(T_ix)|^{p-2}x^*\otimes x(S)\\
 	\Rightarrow \overline{y^*(T_iy)}|y^*(T_iy)|^{p-2}y^*(\mu u)&=&0, ~(\text{since } x^*(x)=1)\\
 	\Rightarrow y^*(T_iy)&=&0,
 \end{eqnarray*}
 which contradicts that $(y,y^*)\in M_{w_p(\mathcal{T})}.$ Therefore, assume that $y^*=\lambda x^*.$  Now,
 \[1=y^*(y)=\lambda x^*(\mu x)=\lambda \mu\quad \Rightarrow \quad \lambda =\overline{\mu},\]
 which implies that $(y,y^*)=(\mu x,\overline{\mu}x^*).$ This proves (ii).\\
 
  (ii) $\Rightarrow$ (i).  In this case, clearly $J_{w_p}(\mathcal{T})$ is singleton. In particular,
  $$J_{w_p}(\mathcal{T})=\bigg\{\bigg(\overline{x_0^*(T_1x_0)}|x_0^*(T_1x_0)|^{p-2}x_0^*\otimes x_0,\ldots,\overline{x_0^*(T_dx_0)}|x_0^*(T_dx_0)|^{p-2}x_0^*\otimes x_0\bigg)\bigg\}$$
   and so $w_p(\cdot)$ is Gateaux differentiable at $\mathcal{T}.$\\
   
   The rest of the part follows from the fact that, $	G(\mathcal{T},\mathcal{S})=\Re(f(\mathcal{S})),$ if $J_{w_p}(\mathcal{T})=\{f\}.$ 
\end{proof}

We would like to remark here that \cite[Th. 3.7]{M2} is a particular case of the previous theorem. As an immediate consequence of Theorem \ref{th-04}, we get the following expression for the Gateaux derivative of the joint numerical radius of a tuple on a Hilbert space.
\begin{cor}\label{cor-06}
	Let $\mathcal{H}$ be a finite-dimensional Hilbert space, and $\mathcal{Z}=(\mathcal{L}(\mathcal{H})^d,w_p(\cdot)).$ Suppose  $w_p(\cdot)$ is Gateaux differentiable at $\mathcal{T}=(T_1,\ldots,T_d)\in \mathcal{Z}.$ Then for each $\mathcal{S}=(S_1,\ldots,S_d) \in \mathcal{Z},$ 
	\begin{eqnarray*}
			G(\mathcal{T},\mathcal{S})
		&=&\frac{1}{w_p(\mathcal{T})^{p-1}}\sum_{i=1}^d\Re\bigg(\overline{\langle T_ix_0,x_0\rangle} |\langle T_ix_0,x_0\rangle|^{p-2}\langle S_ix_0,x_0\rangle \bigg),
	\end{eqnarray*}
where $M_{w_p(\mathcal{T})}=\{(\mu x_0, \overline{\mu} x_0^*):\mu\in \mathbb{F},|\mu|=1\}.$
\end{cor}

\bibliographystyle{amsplain}

\end{document}